\documentclass[12pt,reqno]{amsart}
\usepackage{amsfonts}
\usepackage[a4paper,inner=1in,outer=1in,vmargin=1in,marginparwidth=1in]{geometry}
\usepackage{amsmath,amsthm,amsfonts,amssymb}
\usepackage{amssymb, eucal, amsfonts, amsmath, xypic, latexsym}
\usepackage{bbm}
\usepackage{cite}
\usepackage{color}
\usepackage{enumerate}
\usepackage{pifont}
\usepackage{color}
\usepackage{mathrsfs}
\usepackage{amsthm,indentfirst,bm,fancyhdr,dsfont}
\usepackage{graphicx}

\usepackage[CJKbookmarks=true]{hyperref}

\usepackage{mathrsfs,color}

\makeatletter
\@namedef{subjclassname@2020}{\textup{2020} Mathematics Subject Classification}
\makeatother

\def\C{\mathbb{C}}
\def\N{\mathbb{N}}
\def\Z{\mathbb{Z}}

\newcommand{\op}{\oplus}
\newcommand{\mf}{\mathfrak}
\newcommand{\dis}{\displaystyle}
\newcommand{\bgop}{\bigoplus}
\newcommand{\ot}{\otimes}
\newcommand{\al}{\alpha}
\newcommand{\mcal}{\mathcal}
\newcommand{\be}{\beta}

\newcommand{\wtil}{\widetilde}
\newcommand{\ov}{\overline}
\newcommand{\la}{\lambda}

\newcommand{\what}{\widehat}
\newcommand{\bs}{\boldsymbol}

\newtheorem{theo}{Theorem}[section]

\newtheorem{lemm}[theo]{Lemma}
\newtheorem{remark}[theo]{Remark}
\newtheorem{coro}[theo]{Corollary}
\newtheorem{prop}[theo]{Proposition}

\numberwithin{equation}{section}
\allowdisplaybreaks
\begin{document}
\title{Representations of affine Nappi-Witten Lie algebras over polynomial algebras}

\author[P. Chakraborty]{Priyanshu Chakraborty}
\address{Priyanshu Chakraborty:School of Mathematical Sciences, Ministry of Education Key Laboratory of Mathematics and Engineering Applications and Shanghai Key Laboratory of PMMP,
		East China Normal University, No. 500 Dongchuan Rd., Shanghai 200241, China. }
	\address{Priyanshu Chakraborty:	Indian Institute of Science, Bangalore, India}

\email{priyanshuc@iisc.ac.in, priyanshuc437@gmail.com}
 \author[S. Tantubay ]{Santanu Tantubay }
\address{S. Tantubay: Shenzhen International Center for Mathematics, Southern University of Science and Technology, Shenzhen, China  }
\email{santanu@sustech.edu.cn , 1mathsantanu@gmail.com}



\date{}

\maketitle
\begin{abstract}
In this paper, we study the representation theory of affine Nappi-Witten Lie algebra $\widehat{H_4}$ corresponding to the Nappi-Witten Lie algebra $H_4$. We completely classify  Cartan-free modules of rank one for the Nappi-Witten Lie algebra $H_4$. With the help of Cartan free $H_4$ modules we classify Cartan-free modules of rank one over affine Nappi Witten Lie algebras. We also give a necessary and sufficient condition for these modules to be irreducible. Finally as an application we classify Cartan  free modules of rank one for affine-Virasoro Nappi-Witten Lie algebras.
\end{abstract}

\section{Introduction}
 There are numerous uses for two-dimensional conformal field theories (CFTs) in mathematics and physics. A significant category of CFTs include Wess-Zumino-Novikov-Witten (WZNW) models, studied in \cite{EW}. Nappi and Witten showed in \cite{NW} that a WZNW model is based on a central extension of the two-dimensional Euclidean group that describes the homogeneous four-dimensional space-time corresponding to a gravitational plane wave. The corresponding Lie algebra $H_4$ is known as Nappi Witten Lie algebra. Nappi Witten Lie algebra can be realized as a Lie algebra of the complex four dimensional Lie group $G(0,1)$ which consists of all $4 \times 4$ matrices of the form
  $$  \begin{pmatrix}
1 & ce^{\tau} & a & \tau \\
0 & e^{\tau} & b & 0 \\ 
0 & 0 & 1  & 0\\
0 & 0 & 0& 1 \\
\end{pmatrix} $$
 with $a,b,c, \tau \in \C$.
 This Lie algebra is neither abelian nor semisimple. Nappi-Witten Lie algebra has non-degenerate symmetric bilinear form which is $H_4$ invariant. Therefore with a method similar to affine Kac-Moody Lie algebra, one can affinize Nappi-Witten Lie algebra. The central extension of loop of $H_4$ is known as affine Nappi-Witten Lie algebra. Moreover, the natural action of Virasoro algebra on $\C[t,t^{-1}]$ induces afiine-Virasoro Nappi-Witten Lie algebra structure (see 2.4). Our aim to investigate a class of irreducible representations for Nappi-Witten and affine Nappi-Witten Lie algebras arising from Cartan free modules.\\
 The study of the representation theory of the
 Nappi-witten Lie algebra $H_4$ was started in \cite{MW}. Later a systematic study of representations of $\what H_4$ was initiated in \cite{BJP}, where the authors studied the structures of generalized Verma modules and gave vertex
 operator algebra constructions. In \cite{JW}, the authors classified the simple restricted modules for the affine Nappi-Witten algebra. The category of weight modules with
 finite dimensional weight spaces over the Nappi-Witten vertex operator algebra was studied in \cite{AKD},
 in particular the simple modules were classified and the characters of these simple modules were computed. The imaginary Verma modules and irreducible weight modules with finite weight spaces for the twisted affine Nappi-witten Lie algebra were studied in \cite{XC} and \cite{XC2} respectively. \\
Let $\mathfrak{a}$ be any Lie algebra over $\C$ and $\mathfrak{b}$ be any finite-dimensional abelian subalgebra of $\mathfrak{a}$. Denote $\mathcal{M}(\mathfrak{a},\mathfrak{b})$ by the full subcategory of $U(\mathfrak{a})$-modules consisting of objects whose restriction to $U(\mathfrak{b})$ is a free module of rank $1$, i.e.,
$$\mathcal{M}(\mathfrak{a}, \mathfrak{b})=\{M \in U(\mathfrak{a})-\hbox{Mod}| \hbox{Res}_{U(\mathfrak{b})}^{U(\mathfrak{a})}M \cong_{U(\mathfrak{b})}  U(\mathfrak{b})\},$$
here $U(\mf a), U(\mf b)$ denote the universal enveloping algebra of $\mf a, \mf b$ respectively.
 We call these modules as Cartan free modules with Cartan $\mf b$. In general, Cartan free modules induces non-weight modules.

In recent days, non-weight simple modules for Lie algebras are studied by several authors, for instance see \cite{N1,N2,CH,JZ,PC,XZ,CTZ2,TC}. A class of non-weight modules for finite-dimensional simple Lie algebra of type $A_l$ is studied by J. Nilsson in \cite{N1}, which are known as Cartan-free modules. In \cite{N2}, the author showed that Cartan-free modules exist only for type $A_l(l\geq 1)$ and type $C_l(l\geq 2)$.   For Witt algebra such classification was done in \cite{TZ}. It should be mentioned that in \cite{CTZ1} authors studied Cartan free module affine Kac-Moody Lie algebras, which encourage us to study this type of modules for affine Nappi-Witten Lie algebras. Recently non-weight modules for affine-Virasoro Nappi-Witten Lie algebras has been studied in \cite{CX}.
In this paper, we will study Cartan free modules for Nappi-Witten, affine Nappi-Witten, affine-Virasoro Nappi-Witten Lie algebras. Our results recover result of \cite{CX}.\\

The paper is organised as follows. In Section \ref{NP}, we recall definitions of Nappi-Witten Lie algebra, affine Nappi-Witten Lie algebra as well as affine-Virasoro Nappi-Witten Lie algebra. In Section \ref{NW}, we consider the Nappi-Witten Lie algebra, defined as $H_4=\C p\oplus \C q\oplus \C r\oplus \C s$ with brackets given by 2.1. We classify all possible Cartan-free modules for $H_4$ with Cartan $\mf h=\C s$. 
\begin{theo}
(1).  $ \mathcal{M}(H_4,\mathfrak{h})=\{M_{(g,0)},M_{(0,g)},M_{(h,b)},M_{(b,h)},M_{(a,b)}, M_0: g,h \in \C[s],\;a, b\in \C^*,\; \text{deg}(h)=1\}$, actions of $H_4$ on the modules of $\mathcal{M}(H_4,\mathfrak{h})$ are given by (3.1)-(3.6).\\
(2). $M_{(g,0)}$ and $M_{(0,g)}$ are irreducible iff $g$ is a non-zero constant polynomial.\\
(3). $M_{(h,b)}$, $M_{(b,h)}$ and $M_{(a,b)}$ are irreducible.
\end{theo}
 In Section \ref{ANW}, we consider the affine Nappi-Witten Lie algebra $\widehat{H_4}=H_4\otimes  \C[t^{\pm1}]\oplus \C K \oplus \C d$ with brackets given by 2.2. Let $\what {\mf h}= \C s \op \C d$ be the Cartan subalgebra of $\widehat{H_4}$. For $\alpha\in \C^*$, a sequence of complex numbers $\bs{\beta}=\{\beta_i:i\in \Z,\; \beta_0=0\}$ and a $M \in \mcal M(H_4, \mf h)$ define $\what{M}(\alpha, \bs \beta)=M\otimes \C[d]$ on which action of $\widehat{H_4}$ is given by (4.1).\\
 Also for a sequence of functions $\mathbf{f}=\{f_k(s): f_k(s)\in \C[s],\; f_0(s)=s , \, k \in \Z\}$, define a $\what H_4$-module structure on $\what{M}(\mathbf{f})=\C[s,d]$ on which action of $\widehat{H_4}$ is given by (4.2).\\
 Let us denote these modules for $\what {H_4}$ as $\what M_{(g,0,\al,\bs\be)}$, $\what M_{(0,g,\al,\bs\be)}$, $\what M_{(h,b,\al,\bs\be)}$, $\what M_{(b,h,\al,\bs\be)}$ and $\what M_{(a,b,\al,\bs\be)}$ when corresponding Cartan free rank one $H_4$ modules are $ M_{(g,0)}$, $ M_{(0,g)}$, $ M_{(h,b)}$, $ M_{(b,h)}$ and $M_{(a,b)}$ respectively. Then we have the following theorem.
 \begin{theo}
   (1).  $\mathcal{M}(\what{H_4},\what {\mathfrak{h}})=\{\what M_{(g,0,\al,\bs\be)},\what M_{(0,g,\al,\bs\be)},\what M_{(h,b,\al,\bs\be)} , \what M_{(b,h,\al,\bs\be)}, \what M_{(a,b,\al,\bs\be)}, \what{M}(\mathbf{f}) : g,h \in \C[s],\;a, b\in \C^*,\; \text{deg}(h)=1,\,  {\bf f}=(f_k(s))_{k \in \Z}, \, f_0(s)=s\}$, actions of $\what{H_4}$ on the modules of $\mathcal{M}(\what{H_4},\what {\mathfrak{h}})$ are given by (4.1) and (4.2).  \\
   (2). 
$\what M_{(h,b,\al,\bs\be)}$, $\what M_{(b,h,\al,\bs\be)}$ and $\what M_{(a,b,\al,\bs\be)}$ are irreducible $\what H_4$ module.\\
(3). $\what H_4$-modules $\what M_{(g,0,\al,\bs\be)}$ and $\what M_{(0,g,\al,\bs\be)}$ are irreducible iff $g$ is non-zero constant. \\
(4). $\what{M}(\mathbf{f})$ is reducible. Furthermore, if $f_k(s) \neq 0$ for some $k \neq 0$, then there is a chain of submodules:
$$ <1> \supset <s> \supset <s^2> \supset ....   $$
such that $\frac{<s^m>}{<s^{m+1}>}$ is irreducible for all $m \geq 0$, where $<f(s)>$ denote the ideal generated by $f(s)$ in $\C[s,d].$
 \end{theo}
  Finally as an application of our results we recover the Cartan free modules for affine-Virasoro Nappi-Witten Lie algebras which was obtained in \cite{CX}.

\section{Notation and Preliminaries}\label{NP}
Throughout this paper, $\Z$, $\C$, and $\C^*$ denote the sets of integers, complex numbers, and nonzero complex numbers, respectively.  For a Lie algebra $\mathfrak{a}$, we denote the universal enveloping algebra of $\mathfrak{a}$ as ${U}(\mathfrak{a})$. All the vector spaces, algebras, and tensor products are over $\C$, unless it is specified.

\subsection{Nappi-Witten Lie algebras}
 The Nappi-Witten Lie algebra $H_4$ is a four-dimensional vector space
 $$H_4=\C p\oplus \C q\oplus \C r\oplus \C s$$
together with the Lie brackets
$$[p,q]=r,\ [s,p]=p,\; [s,q]=-q, [r,.]=0. $$
Let $\mathfrak{h}=\C s$ be a Cartan subalgebra of $H_4$.  In the next section, we will describe a class of modules for $H_4$ with respect to this Cartan subalgebra.\\
Let $(,)$ be a symmetric bilinear form on $H_4$ defined by 
$$(p,q)=1,\; (r,s)=1,\; \text{otherwise}, (\;,\;)=0.$$
It is easy to see that $(\;, \;)$ is a non-degenerate $H_4$-invariant symmetric bilinear form on $H_4$.

\subsection{Affine Nappi-Witten Lie algebras.}\label{Aff}
 Let $\C[t^{\pm1}]$ be the Laurent polynomial ring over $\C$.  Now we consider the space 
 $$\widehat{H_4}=H_4\otimes  \C[t^{\pm1}]\oplus \C K \oplus \C d$$
 with the Lie brackets 
 $$[h_1\otimes t^m,h_2\otimes t^n]=[h_1,h_2]\otimes t^{m+n}+m(h_1,h_2)\delta_{m+n,0}K,\; [\widehat{H_4}, K]=0, \; [d, h_1\otimes t^m]=mh_1\otimes t^m,$$
 where $h_1,h_2 \in H_4, \; m,n \in \Z$.\\
 Throughout the paper, we will denote the element $h_1\otimes t^n$ by $h_1(n)$.

 \subsection{Virasoro Lie algebras}
By definition, the Virasoro algebra $Vir:=\C\{d_m, K: m\in \Z\}$, with the brackets
$$[d_m,d_n]=(n-m)d_{m+n}+\delta_{m+n,0}\frac{m^3-m}{12}K,\; [d_m, K=0],$$
for all $m,n\in \Z$. 
 
\subsection{Affine-Virasoro Nappi-Witten Lie algebras}
 We define the affine-Virasoro Nappi-Witten Lie algebra as follows. The underlying vector space is given by
 $$\overline{H_4}=H_4\otimes  \C[t^{\pm1}]\oplus \C K \; \dis{\bgop_{m \in \Z}} \C d_m,$$
and the Lie brackets are given by:
$$[h_1\otimes t^m,h_2\otimes t^n]=[h_1,h_2]\otimes t^{m+n}+m(h_1,h_2)\delta_{m+n,0}{K},\; [\overline{H_4}, {K}]=0, \; [d_m, h_1\otimes t^n]=nh_1\otimes t^{m+n},$$
 $$[d_m,d_n]=(n-m)d_{m+n}+\delta_{m+n,0}\frac{m^3-m}{12}{K}.$$

\section{Cartan free modules over Nappi-Witten Lie algebras}\label{NW}
 In this section, we will discuss Cartan free modules of rank one over Nappi-Witten Lie algebras. Recall that we consider a Cartan subalgebra $\mathfrak{h}=\C s$ for $H_4$. Then ${U}(\mathfrak{h})$ is the polynomial algebra $\C[s]$. We define an automorphism $\tau:{U}(\mathfrak{h}) \rightarrow {U}(\mathfrak{h})$ by $\tau (s)=s-1$.
 \begin{lemm}\label{deg}
  Let $x\in \C[s]$ be any non-constant polynomial. Then $\text{deg}(\tau (x)-x)=\text{deg} (x)-1$.   
 \end{lemm}
 \begin{proof}
     It is easy to prove.
 \end{proof}
Now we define certain class of modules for $H_4$ on the polynomial algebra $\C[s]$. For any $g\in \C[s]$, we define $M_{(g,0)}=\C[s]$ on which action of $H_4$ is given by:
 \begin{equation}
 \begin{cases}
   p.x=\tau (x)g,\\
   q.x=0,\\
   r.x=0,\\
   s.x=sx,
 \end{cases}
\end{equation}
where $x\in \C[s]$.\\

Similarly we define $M_{(0,g)}=\C[s]$ on which action of $H_4$ is given by:
\begin{equation}
 \begin{cases}
   p.x=0,\\
   q.x=\tau^{-1}(x)g,\\
   r.x=0,\\
   s.x=sx,
 \end{cases}
\end{equation}
where $x\in \C[s]$.\\
Now suppose $h(s)=a_1s+a_2$ be a non-zero one-degree polynomial in $\C[s]$ and $b\in \C^*$. We define $M_{(h,b)}=\C[s]$ on which action of $H_4$ is given by:
 \begin{equation}
 \begin{cases}
   p.x=\tau (x) h(s),\\
   q.x=\tau^{-1}(x) b,\\
   r.x= -a_1b,\\
   s.x=sx,
 \end{cases}
\end{equation}
where $x\in \C[s]$.\\
Similarly we define $M_{(b,h)}=\C[s]$ on which action of $H_4$ is given by:
\begin{equation}
 \begin{cases}
   p.x=\tau (x) b,\\
   q.x=\tau^{-1}(x) h(s),\\
   r.x= -a_1b,\\
   s.x=sx,
 \end{cases}
\end{equation}
where $x\in \C[s]$.

Let $a,b$ be two non-zero constant and define $M_{(a,b)}=\C[s]$ on which action of $H_4$ is given by:
 \begin{equation}
 \begin{cases}
   p.x=\tau (x)a,\\
   q.x=\tau^{-1}(x) b,\\
   r.x= 0,\\
   s.x=sx,
 \end{cases}
\end{equation}
where $x\in \C[s]$.\\
Define one another class of module for $H_4$ over the vector space $M_0=\C[s]$ on which action of $H_4$ is given by:
\begin{equation}
 \begin{cases}
   p.x=q.x=r.x=0,\\
    s.x=sx,
 \end{cases}
\end{equation}
where $x\in \C[s]$.\\

It is easy to see that with the above actions $M_{(g,0)}$, $M_{(0,g)}$,  $M_{(h,b)}$, $M_{(b,h)}$, $M_{(a,b)}$ and $M_0$  becomes $H_4$-module, which are Cartan free modules of rank one.\\

Let us define a linear map $\eta:H_4 \to H_4$ by sending
\[
   p \mapsto -q, \hspace{1cm}
   q \mapsto p, \hspace{1cm}
   r \mapsto r, \, \hspace{1cm}
   s \mapsto -s. 
 \]
 It is easy to see that $\eta$ is an automorphism of $H_4$. Now we twist the modules $M_{(g,0)}$ and $M_{(h,b)}$ by the above automorphsim and denote the twisted modules by  $M^t_{(g,0)}$ and $M^t_{(h,b)}$ respectively. Then we have the isomorphism of modules $M^t_{(g(s),0)}\cong M_{(0,g(-s))}$  and $M^t_{(h(s),b)} \cong M_{(-b,h(-s))}$ under the map $f(s) \mapsto f(-s)$ (in both cases).  \\
 
 Now we are going to prove that the above class of modules exhausts all $U(\mf h)$ free modules for $H_4$. The following Lemma easily follows from the Lie brackets of $H_4.$
\begin{lemm}
    Let $M\in \mathcal{M}(H_4,\mathfrak{h})$. Then for any $x\in M$ we have: 
    \begin{equation}
        \begin{cases}
            p.x=\tau (x) p.1,\\
            q.x=\tau^{-1}(x) q.1,\\
            r.x=x r.1,\\
            s.x=sx
        \end{cases}
    \end{equation}
\end{lemm}

\begin{lemm}\label{zeroca}
    Let $M\in \mathcal{M}(H_4,\mathfrak{h})$. 
    \begin{enumerate}
        \item If $p.1=0$ or $q.1=0$ then $r.1=0$. In these cases $M$ will be isomorphic to $M_{(g,0)}$ or $M_{(0,g)}$ for some $g\in \C[s]$.
        \item If $p.1=q.1=0$, then $M$ will be isomorphic to $M_0.$
    \end{enumerate}
\end{lemm}
\begin{proof}
    There is nothing to prove for (2). Suppose $p.1=0$. Then we have $r.1=p.q.1-q.p.1=\tau(q.1) p.1-q.p.1=0$. Similarly, we prove the other part.
\end{proof}
\begin{lemm}\label{cen}
 Let $M\in \mathcal{M}(H_4,\mathfrak{h})$. If $p.1\neq 0$ or $q.1\neq 0$, then we have $r.1\in \C$,   
\end{lemm}
 \begin{proof}
     Suppose $p.1\neq 0$. Now $[p,r].1=0$, which will give us $(\tau (r.1)-r.1)p.1=0$. This will imply that $r\in \C$ by Lemma \ref{deg}. Similarly, we can prove the other part of the Lemma.
\end{proof}
With the thanks to above three Lemmas, we only need to consider the case $p.1\neq 0$ and $q.1\neq 0$. From Lemma \ref{cen} we have $r.1\in \C$.\\
\begin{prop}\label{nonzeroca}
 Let $M\in \mathcal{M}(H_4,\mathfrak{h})$ with $p.1$ and $q.1$ are non-zero. Then $M\cong M_{(h,b)}$ or $M_{(b,h)}$ for some $h(s)=a_1s+a_2\in \C[s]$ and $b\in \C^*$
\end{prop}
 \begin{proof}
   We know $r.1=p.q.1-q.p.1=\tau (q.1) p.1-\tau^{-1} (p.1) q.1=\tau(\tau^{-1}(p.1)q.1)-\tau^-1(p.1)q.1$. By assumption $\tau^{-1}(p.1)q.1$ is a non-zero polynomial in $\C[s]$, therefore by Lemma \ref{deg}, we have either degree of $\tau^{-1}(p.1)q.1$ is one or $p.1$ and $q.1$ both non-zero constant. \\
   {\bf Case I :} Let $p.1$ and $q.1$ are constant. It is clear that in this case $M \cong M_{a,b}$.\\
   {\bf Case II:} Let one of $p.1$ and $q.1$ is not constant. This will imply that the possibility of $\text{deg}(p.1, q.1)$ is $(1,0)$ or $(0,1)$. If $\text{deg}(p.1, q.1)=(1,0)$, then assume $p.1=a_1s+a_2$ and $q.1=b$ for some $a_1, b\in \C^*$ and $a_2\in \C$. Then $r.1=p.q.1-q.p.1=b(a_1s+a_2)-q.(a_1s+a_2)=b(a_1s+a_2)-b(a_1(s+1)+a_2)=-a_1b$. So $M\cong M_{(h,b)}$. Similarly, we can prove if $\text{deg}(p.1, q.1)=(0,1)$, then $M\cong M_{(b,h)}$.

 \end{proof}
\begin{theo}\label{Th H_4}
    $ \mathcal{M}(H_4,\mathfrak{h})=\{M_{(g,0)},M_{(0,g)},M_{(h,b)},M_{(b,h)},M_{(a,b)}, M_0: g,h \in \C[s],\;a, b\in \C^*,\; \text{deg}(h)=1\}$
\end{theo}
 \begin{proof}
     Follows from Lemma \ref{zeroca} and Proposition \ref{nonzeroca}.
 \end{proof}
 Now we are interested to find the conditions when the modules of $\mathcal{M}(H_4,\mathfrak{h})$ are irreducible.
 \begin{theo}\label{Th 3.6}
   \begin{enumerate}
       \item The modules $M_{(g,0)}$ and $M_{(0,g)}$ are irreducible iff $g$ is a non-zero constant polynomial.
       \item The modules $M_{(h,b)}$, $M_{(b,h)}$ and $M_{(a,b)}$ are irreducible.
      \item For a non-constant polynomial $g(s)=C(s-\al_1) \cdots (s-\al_k)$, the submodules of $M_{(g,0)}$ are given by $\dis{\prod_{r=0}^{l_1}} (s-\al_{i_1}-r) \dis{\prod_{r=0}^{l_2}} (s-\al_{i_2}-r) \cdots \dis{\prod_{r=0}^{l_j}} (s-\al_{i_j}-r) $ for $i_j \in \{1,2, \cdots, k\} $ and $l_j \in \N$. 
   \end{enumerate}  
 \end{theo}
 \begin{proof}
    Let $g$ be a non-zero constant polynomial and $W$ be a non-zero submodule of $M_{(g,0)}.$ Note that due to the action of $s$, it is sufficient to prove that $1 \in W$. Let $f(s)=\dis{\sum_{i=0}^{k}}a_is^{k-i} $ be a non-zero polynomial of degree $k$ in $W$. Then $p.f(s)=f(s-1)g \in W.$ Now consider the vector
   \begin{align*}
       f_1(s)=\frac{1}{g}p.f(s)-f(s) &= \dis{\sum_{i=0}^{k}}a_i(s-1)^{k-i}-\dis{\sum_{i=0}^{k}}a_is^{k-i} \cr
       &=-k a_0s^{k-1} \; (mod \; \dis{\bgop_{i=2}^{k}}\C s^{k-i}).
\end{align*} 
This means $f_1(s) $ is a non-zero polynomial of degree less than $k$ in $W$. Continuing this process we have $1 \in W.$ This proves that $M_{(g,0)}$ is irreducible.\\
On the other hand if $g$ is a polynomial of degree greater equal to 1, then the ideal generated by $g$ is a proper submodule of $M_{(g,0)},$ 
Hence $M_{(g,0)}$ is reducible. Similar proof works for $M_{(0,g)}.$\\
To prove (2) consider a non-zero submodule $W_1$ of $M_{(h,b)}.$ Let $f(s)$ be a degree $k$ non-zero polynomial in $W_1.$ Now consider the vector $\frac{1}{b}q.f(s)-f(s) $ and proceed similarly like (1) to conclude $1 \in W_1.$ This proves that $M_{(h,b)}$ is irreducible. Similarly $M_{(b,h)}$ is irreducible.\\
(3) Let $W$ be a non-zero proper submodule of $M_{(g,0)}$. It is easy to see by the action of $s$ on $M_{(g,0)}$ that $W$ is an ideal of $\C[s].$ Hence $W$ is generated by some non-zero non-constant polynomial $F(s)$, as $\C[s]$ is a PID. Now $p.F(s)=F(s-1)g(s) \in W$, therefore $F(s)$ divides $F(s-1)g$. Therefore all the roots of $F(s)$ comes form roots of $g(s)$ and $F(s-1).$ Let $d(s)$ be the gcd of $F(s)$ and $g(s)$. Then $d(s)$ can not be 1, otherwise $F(s)$ divides $F(s-1)$ and hence $F(s)= constant.$ Therefore $W=M_{(g,0)}$ in this case. So $F(s)$ and $g(s)$ must have a common factor. Now it follows from the fact $F(s)|F(s-1)g(s)$ that, $F(s)= \dis{\prod_{r=0}^{l_1}} (s-\al_{i_1}-r) \dis{\prod_{r=0}^{l_2}} (s-\al_{i_2}-r) \cdots \dis{\prod_{r=0}^{l_t}} (s-\al_{i_t}-r)$, for some $l_1, \cdots, l_t \in \N$ and $\{\al_{i_1}, \cdots \al_{i_t} \} \subseteq \{\al_1, \al_2, \cdots, \al_k \}$.

\end{proof}

\section{Cartan Free modules over affine Nappi-Witten Lie algebras}\label{ANW}

In this section, we will study the categories $\mathcal{M}(\widehat{H_4}, \widehat {\mf h}  )$ for the Lie algebras $\what{H_4}$. Suppose $\C[s,d]$ be the polynomial algebra over two indeterminant $s$ and $d$. Let us consider $\tau,\; \sigma$ be two automorphism of $\C[s,d]$ defined by
\[\tau(s)=s-1,\; \tau (d)=d,\; \sigma(s)=s,\; \sigma (d)=d-1.\]
We see that $\tau$ and $\sigma$ are two commuting automorphisms.\\
Now we construct a class of modules for $\widehat{H_4}$. Suppose $M$ is a Cartan free module over the Nappi-Witten Lie algebra $H_4$. For $\alpha\in \C^*$ and a sequence of complex numbers $\boldsymbol{\beta}=\{\beta_i:i\in \Z,\; \beta_0=0\}$, define $\what{M}(\alpha, \bs \beta)=M\otimes \C[d]$ on which action of $\widehat{H_4}$ is given by:
\begin{equation} \label{actions of mod}
    \begin{cases}
        p\otimes t^k. x(s,d)= \alpha^k \tau \sigma^k(x(s,d))p.1\\
        q\otimes t^k. x(s,d)= \alpha^k \tau^{-1}\sigma^k(x(s,d))q.1\\
        r\otimes t^k. x(s,d)=\alpha^k \sigma^k(x(s,d))r.1\\
        s\otimes t^k. (s,d)=  \sigma^k(x(s,d)) (\alpha^ks+\beta_k)\\
        K.(x(s,d))=0\\
        d.x(s,d)=dx(s,d),
    \end{cases}
    \end{equation}
 where $x(s,d)\in \C[s,d]$ and $p.1, q.1,r.1$ are the actions obtained from the action of $H_4$ on $M$ (i.e, these actions obtained from (3.1)-(3.5) ).\\
Let us denote the above obtained modules for $\what {H_4}$ as $\what M_{(g,0,\al, \bs  \be)}$, $\what M_{(0,g,\al,\bs \be)}$, $\what M_{(h,b,\al,\bs \be)}$, $\what M_{(b,h,\al,\bs \be)}$ and $\what M_{(a,b,\al,\bs \be)}$ when corresponding Cartan free rank one $H_4$ modules are $ M_{(g,0)}$, $ M_{(0,g)}$, $ M_{(h,b)}$, $ M_{(b,h)}$ and $M_{(a,b)}$ respectively. 

 For a sequence of functions $\mathbf{f}=\{f_k(s): f_k(s)\in \C[s],\; f_0(s)=s , \, k \in \Z\}$, define a $\what H_4$-module structure on $\what{M}(\mathbf{f})=\C[s,d]$ on which action of $\widehat{H_4}$ is given by:
 \begin{equation}
    \begin{cases}
        s\otimes t^k. x(s,d)=  \sigma^k(x(s,d)) f_k(s)\\
         d.x(s,d)=dx(s,d),\\
       p\otimes t^k. x(s,d)= q\otimes t^k. x(s,d)= r\otimes t^k. x(s,d)=K.(x(s,d))=0
       
    \end{cases}
    \end{equation}

\begin{prop}
Under actions of (4.1) and (4.2), $\what{M}(\alpha, \bs \beta)$ and $\what{M}(\mathbf{f})$ forms $\what H_4$-module.
\end{prop}
\begin{proof}
Let $k,l \in \Z$ and $x(s,d) \in \C[s,d]$. Now we compute all possible relations to prove that actions in (4.1) and (4.2) are module actions.\\
\begin{align}
 (i). \,\,   [p\ot t^k,q \ot t^l].x(s,d) &=(p\ot t^k.q \ot t^l-q \ot t^l.p\ot t^k).x(s,d) \cr
    & = \al^lp\ot t^k.x(s+1,d-l)q.1 -\al^kq\ot t^l .x(s-1,d-k)p.1  \cr
    &=\al^{k+l}(x(s,d-l-k)\tau(q.1)p.1-x(s,d-l-k)\tau^{-1}(p.1)q.1 ) \cr
    &=\al^{k+l}x(s,d-l-k)(\tau(q.1)p.1- \tau^{-1}(p.1)q.1)\cr
    &=\al^{k+l}x(s,d-l-k)r.1 \cr
    &=r\ot t^{k+l}.x(s,d)\cr
\end{align}

\begin{align} 
(ii). \, \, [p\ot t^k,r \ot t^l].x(s,d)&=(p\ot t^k.r \ot t^l-r \ot t^l.p\ot t^k).x(s,d) \cr
    &=\al^lp\ot t^k.x(s,d-l)r.1 -\al^kr\ot t^l .x(s-1,d-k)p.1   \cr
    &=\al^{k+l}(x(s-1,d-l-k)\tau(r.1)p.1-x(s-1,d-l-k)p.1r.1)  \cr
    &=0,\cr 
\end{align}
 since $\tau(r.1)=r.1$.    
\begin{align}
    (iii). \,  [p\ot t^k,s \ot t^l].x(s,d)&=(p\ot t^k.s \ot t^l-s \ot t^l.p\ot t^k).x(s,d) \cr
    &=p\ot t^k.x(s,d-l)(\al^ls.1+\be_l) -\al^ks\ot t^l .x(s-1,d-k)p.1   \cr
    &=\al^kx(s-1,d-l-k)[\tau(\al^ls.1 +\be_l)p.1-\sigma(p.1)(\al^ls.1+\be_l)] \cr
    &=\al^kx(s-1,d-l-k)[(\al^l(s-1)+\be_l)p.1-(\al^ls+\be_l) p.1)\cr
    &=-\al^{k+l}x(s-1,d-l-k)p.1=-p\ot t^{k+l}.x(s,d)\cr
\end{align}    
    
\begin{align}
(iv). \, \,[q\ot t^k,r \ot t^l] &=(q\ot t^k.r \ot t^l-r \ot t^l.q\ot t^k).x(s,d) \cr
    &=\al^lq\ot t^k.x(s,d-l)r.1 -\al^kr\ot t^l .x(s+1,d-k)q.1   \cr
    &=\al^{k+l}(x(s+1,d-l-k)\tau^{-1}(r.1)q.1-x(s+1,d-l-k)q.1r.1)  \cr
    &=0=[q\ot t^k, r\ot t^l].x(s,d), \cr 
    \end{align}
    since $\tau^{-1}(r.1)=r.1$.
    \begin{align}
    (v). \, \, [q\ot t^k,s \ot t^l] &=(q\ot t^k.s \ot t^l-s \ot t^l.q\ot t^k).x(s,d) \cr
    &=q\ot t^k.x(s,d-l)(\al^ls.1+\be_l) -\al^ks\ot t^l .x(s+1,d-k)q.1   \cr
    &=\al^kx(s+1,d-l-k)[\tau(\al^ls.1 +\be_l)p.1-\sigma(q.1)(\al^ls.1+\be_l)]  \cr
    &=\al^kx(s+1,d-l-k)[(\al^l(s+1)+\be_l)q.1-(\al^ls+\be_l) q.1)\cr
    &=\al^{k+l}x(s+1,d-l-k)q.1=q\ot t^{k+l}.x(s,d)\cr
     \end{align}
For the action (4.1) other relations are easy to check. Now we check the module relations for the actions of (4.2).\\
\begin{align}
(s\ot t^k.s \ot t^l-s \ot t^l.s\ot t^k).x(s,d) 
&=x(s,d-k-l)f_k(s)f_l(s)-x(s,d-k-l)f_k(s)f_l(s)\cr
&=0=[s\ot t^k,s \ot t^l].x(s,d)
\end{align}
Again,\\
\begin{align}
[d,s\ot t^l].x(s,d)&=(d.s \ot t^l-s \ot t^l.d).x(s,d) \cr
&= dx(s,d-l)f_l(s)-(d-l)x(s,d-l)f_l(s)=lx(s,d-l)f_l(s)\cr
&=ls\ot t^l.x(s,d).
 \end{align}

It is easy to see that, due to the zero actions of $p \ot t^k, q\ot t^k, r\ot t^k$ in (4.2) all other relations satisfies trivially. This completes the proof.

\end{proof}

 \begin{theo}\label{Main theo}
$\mathcal{M}(\what{H_4}, \what{\mathfrak{h}})=\{\what M_{(g,0,\al,\bs \be)},\what M_{(0,g,\al,\bs \be)},\what M_{(h,b,\al, \bs \be)} , \what M_{(b,h,\al,\bs \be)}, \what M_{(a,b,\al, \bs \be)}, \what{M}(\mathbf{f}) : g,h \in \C[s],\;a, b\in \C^*,\; \text{deg}(h)=1,\,  \alpha \in \C^*,\;\bs{\beta} \in \C^{\Z},\; \mathbf{f}\in \C[s]^{\Z}, \, f_0(s)=s\}$, where $A^\Z$ denote the set of all functions from $\Z$ to $A.$

 \end{theo}
 Suppose $M$ is a Cartan-free module over $\widehat{H_4}$, then as vector space $M=\C[s,d]$.
 Assume that $p\otimes t^k.1 =g_k(s,d),\; q\otimes t^k.1 =h_k(s,d),\;   s\otimes t^k=f_k(s,d)\in \C [s,d]$, where $k\in \Z$. Note that if $g_k(s,d)$ (or $h_k(s,d)$ ) is zero for some $k\in \Z$, then $g_0(s,d)$  (or $h_0(s,d)$) is also zero polynomial.  For notational convenience we use $deg_d(f)$ and $deg_s(f)$ to denote the degree of $d$
 and $s$ of a polynomial $f \in \C[s,d]$. Now we proceed to prove Theorem 4.2, we start with the following lemma.
 

\begin{lemm}\label{lem degdf}
    Let $f_k(s,d)$ be as above. If for some $k\in \Z$, $f_k$ is a non-zero polynomial, then we will have $deg_d(f_k)=0$.
\end{lemm}
\begin{proof}
   We see that for all $k \neq 0,$ $0=[s\otimes t^k,s\otimes t^{-k}].1=f_k(s,d)f_{-k}(s,d-k)-f_k(s,d+k)f_{-k}(s,d)$, hence $H(s,d)=H(s, d+k)$, where $H(s,d)=f_k(s,d)f_{-k}(s,d-k)$. So we get $deg_{d}(f_k)=deg_d(f_{-k})=deg_d(H)=0$.  
\end{proof}
\begin{lemm}\label{deg f_k}
Suppose $g_0$ (or $h_0$) is a non-zero polynomial, then $deg_d(f_k)=0$ and $deg_s(f_k)\leq 1$ for all $k\in \Z$.
\end{lemm}
\begin{proof}
 We will prove for the case $g_0\neq 0$. Similarly one prove for $h_0\neq 0$.\\
We see that $f_k\neq 0$ for all $k\in \Z$, otherwise $g_0=p.1=[s\otimes t^k,p\otimes t^{-k}].1=0$. Now the first part of the statement follows from Lemma \ref{lem degdf}.\\
 We have $g_k(s,d)=[s\otimes t^k,p].1=g_0(s,d-k)f_k(s)-f_k(s-1)g_0(s,d)$ and putting this value in the equation $0=[p,p\otimes t^k].1$, we get:
    \[0=g_k(s-1,d)g_0(s,d)-g_k(s,d)g_0(s-1,d-k)\]
    \[=\{g_0(s-1,d-k)f_k(s-1)-f_k(s-2)g_0(s-1,d)\}g_0(s,d)-\]
    \[\{g_0(s,d-k)f_k(s)-f_k(s-1)g_0(s,d)\}g_0(s-1,d-k).\]
    This will give us
    \[2f_k(s-1)g_0(s-1,d-k)g_0(s,d)=f_k(s)g_0(s-1,d-k)g_0(s,d-k)+f_k(s-2)g_0(s,d)g_0(s-1,d).\]
    Let $g_0(s,d)=\dis{\sum_{i=0}^m}a_i(s)d^i. $ Now putting the value of $g_0(s,d)$ in the above equation we have,\\
    $2f_k(s-1)\dis{\sum_{i=0}^m}a_i(s-1)(d-k)^i\dis{\sum_{i=0}^m}a_i(s)d^i=f_k(s)\dis{\sum_{i=0}^m}a_i(s-1)(d-k)^i\dis{\sum_{i=0}^m}a_i(s)(d-k)^i $
    $$ +f_k(s-2)\dis{\sum_{i=0}^m}a_i(s-1)d^i\dis{\sum_{i=0}^m}a_i(s-1)d^i.  $$
    
    Now comparing the coefficient of $d^{2m}$ in the above equation we will get
    \[2f_k(s-1)=f_k(s)-f_k(s-2)\]
    and hence $deg_s(f_k)\leq 1$. Similarly, we can prove that if $q.1\neq 0$, then $deg_d(f_k)=0$ and $deg_s(f_k)\leq 1$.

\end{proof}
 Therefore with the help of Lemma 4.4, we assume $f_k=\alpha_k s+\beta_k$, where $\alpha_k, \beta_k\in \C$ for all $k \in \Z$. Note that $\alpha_0=1$ and $\beta_0=0$. We also assume $\alpha_1=\alpha$ and $\beta_1=\beta$.
\begin{lemm}\label{lem d=0}
Let $g_k(s,d),\; h_k(s,d),\; f_k$ be as above. If $g_0$ or $h_0$ is non-zero polynomial, then we have $\alpha_{-1}=\alpha^{-1}$ and
\begin{enumerate}
    \item $deg_d(g_0(s,d))=0$ if $g_0\neq 0,$
    \item $deg_d(h_0(s,d))=0$ if $h_0\neq 0.$
\end{enumerate} .
\end{lemm}
\begin{proof}
 We will prove for $g_0\neq 0$, one proves the other case similarly.
    We assume $g_0(s,d)=\dis{\sum_{i=0}^m}a_i(s)d^i$. We need to prove $m=0$.\\
    From the relation $[s\otimes t, p].1=p\otimes t.1$ we say that
    \begin{equation}\label{eqn val g1}
        g_1(s,d)=(\alpha s+\beta)(g_0(s,d-1)-g_0(s,d))+\alpha g_0(s,d).
    \end{equation}
    Again from the  relation $[s\otimes t^{-1}, p\otimes t].1=p.1$ and the equation (\ref{eqn val g1}) we will have
    \begin{equation}\label{eqn val g0}
    \begin{split}
    (\alpha_{-1}s+\beta_{-1})\{(\alpha s+\beta)(g_0(s,d)-g_0(s,d+1))+\alpha g_0(s,d+1)\}-\\ (\alpha_{-1}(s-1)+\beta_{-1})\{(\alpha s+\beta)(g_0(s,d-1)-g_0(s,d))+\alpha g_0(s,d)\}=g_0(s,d). \end{split}
    \end{equation}
Note that coefficient of $d^m$ in both $g_0(s,d)-g_0(s,d+1)$ and $g_0(s,d-1)-g_0(s,d)$ are zero. Therefore comparing the coefficients of $d^m$ on both sides of the equation (\ref{eqn val g0}), we get
    \[ \alpha (\alpha_{-1}s+ \beta_{-1})a_m(s)-\alpha(\alpha_{-1}(s-1)+\beta_{-1})a_m(s)=a_m(s)\]
    Hence we have $\alpha_{-1}=\alpha^{-1}$.\\
    Now from the relation $[s\otimes t^{-1}, p].1=p\otimes t^{-1}.1$ we have
    \begin{equation}\label{eqn val g-1}
       g_{-1}(s,d)=(\alpha_{-1}s+\beta_{-1})(g_0(s,d+1)-g_0(s,d))+\alpha_{-1}g_0(s,d).
    \end{equation}
    Now from the relation $[p\otimes t,p\otimes t^{-1}].1=0$, we get
    \begin{equation}\label{eqn p}
    g_{-1}(s-1,d-1)g_1(s,d)=g_{-1}(s,d)g_1(s-1,d+1)
\end{equation}
Now we compute the coefficient of $d^{2m-1}$ on both side of equation (\ref{eqn p}) with the help of equation (\ref{eqn val g-1}) and equation (\ref{eqn val g1}). Now\\

$g_{-1}(s-1,d-1)g_1(s,d)$\\

$=\{(\alpha_{-1}(s-1)+\beta_{-1})(g_0(s-1,d)-g_0(s-1,d-1))+\alpha_{-1}g_0(s-1,d-1)\} \times   \{(\alpha s+\beta)(g_0(s,d-1)-g_0(s,d))+\alpha g_0(s,d)\}$\\

$=\{(\alpha_{-1}(s-1)+\beta_{-1})(ma_m(s-1)d^{m-1}+\mcal O(d^{m-2}))+\alpha_{-1}(a_m(s-1)d^m-ma_m(s-1)d^{m-1}+a_{m-1}(s-1)d^{m-1}+\mcal O(d^{m-2})\} \times
\{(\alpha s+\beta)(-ma_m(s)d^{m-1}+ \mcal O(d^{m-2}))+\alpha(a_m(s)d^m+a_{m-1}(s)d^{m-1}+\mcal O(d^{m-2})\}$\\

$=\{\al_{-1}a_m(s-1)d^m+[ma_m(s-1)(\alpha_{-1}(s-1)+\beta_{-1})-\al_{-1}ma_m(s-1)+\al_{-1}a_{m-1}(s-1)]d^{m-1}+\mcal O(d^{m-2})\} \times\{ \alpha a_m(s)d^m+[-(\alpha s+\beta)ma_m(s)+a_{m-1}(s)]d^{m-1} +\mcal O(d^{m-2})  \}$\\

From the above we have coefficient of $d^{2m-1}$ of $g_{-1}(s-1,d-1)g_1(s,d)$ equal to \\

$a_{-1}a_m(s-1)[\al a_{m-1}(s)-m(\al s+\be)a_m(s)] + \al a_m(s)[ma_m(s-1)(\al_{-1}(s-1)+\be_{-1})-\al_{-1}ma_m(s-1)+\al_{-1}a_{m-1}(s-1)].$ \\

In a similar way we have coefficient of $d^{2m-1}$ of $g_{-1}(s,d)g_1(s-1,d+1)$ equal to\\

$\alpha a_m(s-1)[\alpha_{-1}a_{m-1}(s)+m(\alpha_{-1}s+\beta_{-1})a_m(s)]+\alpha_{-1}a_m(s)[m\alpha a_m(s-1)+\alpha a_{m-1}(s-1)-m(\alpha (s-1)+\beta)a_m(s-1)].$\\

Now equating both coefficient of $d^{2m-1}$ we get \\

$\alpha_{-1}a_m(s-1)\{\alpha a_{m-1}(s)-m(\alpha s+\beta)a_m(s)\}+\alpha a_m(s)\{ma_m(s-1)(\alpha_{-1}(s-1)+\beta_{-1})-\alpha_{-1}(m a_m(s-1)-a_{m-1}(s-1)\}$\\
$=\alpha a_m(s-1)\{\alpha_{-1}a_{m-1}(s)+m(\alpha_{-1}s+\beta_{-1})a_m(s)\}+\alpha_{-1}a_m(s)\{m\alpha a_m(s-1)+\alpha a_{m-1}(s-1)-m(\alpha (s-1)+\beta)a_m(s-1)\}.$\\
$\implies 4m\alpha \alpha_{-1} a_m(s)a_m(s-1)=0 \implies m=0$. This completes the proof.

\end{proof}

\begin{prop}\label{prop 4.6}
    Let $M\in \mathcal{M}(\widehat{H_4},\hat{\mathfrak{h}})$, then $p\otimes t^k.1=\alpha ^kp.1$, $q\otimes t^k.1=\alpha ^kq.1,$ $r\otimes t^k.1=\alpha ^kr.1$ for all $k\in \Z$.
\end{prop}
 
\begin{proof}
From $p\otimes t^k=[s\otimes t^k, p]$ we see that if $p.1=0$, then $p\otimes t^k=0$ for all $k\in \Z$. Now we assume $p.1\neq 0$.
    From Lemma \ref{lem d=0}, we know that $\deg_d(g_0)=0$. Now $p\otimes t.1=[s\otimes t,p].1=(\alpha s+\beta)g_0(s)-g_0(s)(\alpha (s-1)+\beta)=\alpha g_0(s)=\alpha p.1$. Also we know, $p\otimes t^{k+1}=[s\otimes t, p\otimes t^k]$. Now use induction on $k$ for positive integer $k$. Similarly, we prove for negative integers with the property $\alpha_1=\alpha^{-1}$.
    One proves for $q\otimes t^k$ similarly. Now $r\otimes t^k.1=[p\otimes t^k,q].1=[p,q\otimes t^k].1=\alpha^k r.1$.
\end{proof}
{\bf Proof of theorem \ref{Main theo} :} 
\begin{proof}
From Lemma \ref{lem d=0} we have, $p.1$ and $q.1$ are in $\C[s]$. This implies that $\C[s]$ is a rank one $U(H_4)$ submodule of $\C[s,d].$ Therefore we get the actions of $p.1, q.1$ and $r.1$ from Theorem \ref{Th H_4}.\\
{\bf Case I: } Let $p.1 \neq 0$ ( similar case will arise if $q.1 \neq 0$). Then we use the relation $[s\ot t^k,p].1=p \ot t^k.1$ to conclude that $\al_k= \al^k.$\\
{\bf Subcase (i): } Let $p.1=a$ and $q.1=b$, $a, b \in \C^*.$ Then $r.1=0$ and,
\begin{align}
  & r.1+K.1 = [p \ot t, q\ot t^{-1}] =0\cr
& \implies K.1 =0 .
\end{align}  

{\bf Subcase (ii): } Let $p.1=a_1s+a_2$ and $q.1=b$, $a_1,b \in \C^*, a_2 \in \C.$ Then $r.1=-a_1b$ and
\begin{align}
 & K.1+r.1=-a_1b\al_1\al_{-1} \cr
& \implies K.1=0.
    \end{align}  
{\bf Subcase (iii): } Let $p.1=g(s)$ and $q.1=0$. Then clearly we have $K.1=0.$  
Therefore we have $M \cong \what M_{(g,0,\al,\bs \be)}$ or $\what M_{(h,b,\al,\bs\be)}$ or $\what M_{(a,b,\al,\bs\be)}$, thanks to Proposition \ref{prop 4.6}.  In a similar way considering $q.1 \neq 0$ we have $M \cong \what M_{(0,g,\al,\bs\be)}$ or $\what M_{(b,h,\al,\bs\be)}$.  \\
{\bf Case II:} Let $p.1= q.1=0.$ Then clearly $M \cong \what M({\mathbf f})$, with the help of Lemma \ref{lem degdf}. This completes the proof.

\end{proof}
\begin{remark}
    Note that when $q.1=0$ and $p.1 \neq 0$ classification of Cartan free modules for $\what {H_4}$ reduces to the problem of classifying Cartan free modules for the affinization of two dimensional non-abelian Lie algebras. Therefore classifiaction of Cartan free modules of rank one for the Lie algebra $\mf L= <s,p> \ot \C[t,t^{-1}] \op \C K \op \C d$ is given by $\what M(\al, \mathbf \be)$ with actions defined in (\ref{actions of mod}). 
\end{remark}

\begin{theo}
\begin{enumerate}
    \item $\what M_{(g,0,\al,\bs\be)} \cong \what M_{(g_1,0,\al_1,\bs\be_1)}$ iff $g=g_1,\al=\al_1, \bs\be=\bs\be_1$
   \item  $\what M_{(0,g,\al,\bs\be)} \cong \what M_{(0,g_1,\al_1,\bs\be_1)}$ iff $g=g_1,\al=\al_1, \bs\be=\bs\be_1$
   \item  $\what M_{(h,b,\al,\bs\be)} \cong \what M_{(h_1,b_1,\al_1,\bs\be_1)}$ iff $h=h_1,\al=\al_1, \bs\be=\bs\be_1, b=b_1$
   \item  $\what M_{(b,h,\al,\bs\be)} \cong \what M_{(b_1,h_1,\al_1,\bs\be_1)}$ iff $h=h_1,\al=\al_1, \bs\be=\bs\be_1, b=b_1$
    \item $\what M_{(a,b,\al,\bs\be)} \cong \what M_{(a_1,b_1,\al_1,\bs\be_1)}$ iff $a=a_1,\al=\al_1, \bs\be=\bs\be_1, b=b_1$
    \item $\what{M}(\mathbf{f})\cong \what{M}(\mathbf{f'})$ iff $f_k(s) = f'_k(s)$ for all $k \in \Z.$
\end{enumerate}  
\begin{proof}
    We prove (3) and all other follows with the simialr proof. Let $\phi:\what M_{(h,b,\al,\bs\be)} \to \what M_{(h_1,b_1,\al_1,\bs\be_1)}$ be an isomorphism. Then we can see that $\phi(1)$ must be a non-zero scalar. Now we use the relation $\phi(x.1)=x.\phi(1)$, for $x=p, q$. This gives us $h=h_1$ and $b=b_1.$ Then using the relation $\phi(x\ot t^k.1)=x\ot t^k.\phi(1)$ for $x=p,s$ and $k \neq 0$ we will get that $\al=\al_1$ and $\bs\be=\bs\be_1.$ This completes the proof.
\end{proof}

\end{theo}

\begin{theo}\label{Th 4.9}
\begin{enumerate}
    \item 
$\what M_{(h,b,\al,\bs\be)}$, $\what M_{(b,h,\al,\bs\be)}$ and $\what M_{(a,b,\al,\bs\be)}$ are irreducible $\what H_4$ module.
\item $\what H_4$-modules $\what M_{(g,0,\al,\bs\be)}$ and $\what M_{(0,g,\al,\bs\be)}$ are irreducible iff $g$ is  non-zero constant.
\item $\what{M}(\mathbf{f})$ is reducible. Furthermore, if $f_k(s) \neq 0$ for some $k \neq 0$, then there is a chain of submodules:
$$ <1> \supset <s> \supset <s^2> \supset ....   $$
such that $\frac{<s^m>}{<s^{m+1}>}$ is irreducible for all $m \geq 0$, where $<f(s)>$ denote the ideal generated by $f(s)$ in $\C[s,d].$
     \end{enumerate}
\end{theo} 
\begin{proof}
  (1). We prove it for $\what M_{(h,b,\al,\bs\be)}$ (similar proof work for other two modules).  Let $W$ be a non-zero submodule of $\what M_{h,b,\al,\bs\be}$ and $w \in W$ be a non-zero element of $W$ of smallest degree in $d.$ Note that $deg_d(q \ot t.w -f_1q.w) < deg_d(w)$ and $q \ot t.w -f_1q.w$  is a non-zero element of $W$, hence $deg_d(w)=0 $, i.e $w \in M_{h,b}.$ Now observe that $U(H_4)w $ is a non-zero submodule of $M_{h,b}$, so $M_{h,b} \subseteq W.$ Now using the action of $d$ we have $W=M_{h,b}\ot \C[d].$ This completes the proof of (1).\\
  (2) Note that if $g$ is constant the proof of (1) will run through and hence $\what M_{(g,0,\al,\bs\be)}$ is irreducible. On the other hand, if $g$ is non-constant, then corresponding to a factor $s -\al$ of $g$, the ideal generated by $s - \al$ in $\C[s,d]$ forms a proper submodule of $\what M_{(g,0,\al,\bs\be)}$.\\
  To prove (3) observe that, ideal generated by $s^m$ in $\C[s,d]$ forms a proper submodule of $\what{M}(\mathbf{f})$ for all $m \geq 1$. Let $M_m=<s^m>/<s^{m+1}> $  for all $m \geq 0$ and $f_k(s) \neq 0.$ Note that $M_m=span\{ s^m f(d) + <s^{m+1}>: f(d) \in \C[d]\}$. Let $W$ be a non-zero submodule of $M_m.$ Let $x(s,d) +<s^{m+1}> $ be a non-zero element of $W$ of smallest degree in $d$. Then consider $(s\ot t^k-f_k(s)).(x(s,d) +<s^{m+1}>) $, which is a non-zero polynomial and $deg_d((s\ot t^k-f_k(s)).(x(s,d) +<s^{m+1}>)) < deg_d(x(s,d) +<s^{m+1}> ).$ This forces that $W$ contains an elements which is a polynomial in $s$, i.e $s^m +<s^{m+1}>\in W.$ Now considering the action of $d $ on $s^m +<s^{m+1}>$ we have $W=M_m.$ This completes the proof.
  
\end{proof}
\begin{remark}
\begin{enumerate}
    \item[1.] Theorem \ref{Th 3.6} and \ref{Th 4.9} gives us,  Cartan free modules for $\widehat H_4$ induced from $H_4$ is irreducible if and only if corresponding modules for $H_4$ is irreducible. There is a analogous result for irreducibility of Cartan free modules for affine Kac-Moody Lie algebras, Proposition 9, \cite{CTZ1}. 

\item[2.] Unfortunately, we are unable to find the form of all submodules for  $\what M_{(g,0,\al,\bs\be)}$ and $\what M_{(0,g,\al,\be)}$ when $g$ is non-constant. But we have a class of submodules for $\what M_{(g,0,\al,\bs\be)}$ for non-constant $g.$ Let $k \in \N$ and  $s-\al$ be a factor of $g$. Then the ideal of $\C[s,d]$ generated by $\{\dis{\prod_{r=0}^{k}} (s-\al-r) d^k, \dis{\prod_{r=0}^{k-1}}(s-\al-r)d^{k-1}, \cdots , (s-\al) \}  $ forms a submodule of $\what M_{(g,0,\al,\bs\be)}$.  Moreover, it is easy to construct submodules which will arise from the submodules of $H_4$. We believe that all submodules in this case will arise from submodules of $H_4.$
\end{enumerate}
\end{remark}

Now as an application of Theorem \ref{Main theo} we recover the Cartan free modules for Affine-Virasoro Nappi-Witten Lie algebras, which was obtained in \cite{CX}. For this we need the help of Theorem 3.2 of \cite{HCS}. At first we state it here. Let $Vir(0,0)$ be the Lie algebra with basis $\{d_n, W_n,K: n \in \Z\}$ with bracket relations:
$$ [d_n,d_m]=(n-m)d_{m+n}+\delta_{m+n,0}\frac{m^3-m}{12}K$$ 
$$[d_n,W_m]=mW_{m+n}$$
$$[K,.]=0$$
for all $m,n \in \Z.$
\begin{theo}\label{th Vir}
    Let $M$ be a free $U(\C d_0 \op \C W_0)$-module of rank 1 for $Vir(0,0)$. Then $M \cong M(\la,  f)$, where $\la \in \C^*$ and $f$ is a polynomial in $W_0.$ Actions of elements of $Vir(0,0)$ on $U(\C d_0 \op \C W_0) $ are given by
    \begin{align}
        d_m.x(d_0,W_0)=\la^m(d_0+mf(W_0))x(d_0-m,W_0)\\
        W_m.x(d_0,W_0)=\la^mW_0x(d_0-m,W_0)\\
        K.x(d_0,W_0)=0,
    \end{align}
    for all $m \in \Z$, $x(d_0,W_0) \in U(\C d_0 \op \C W_0)$.
\end{theo}
\begin{theo}
  Let $M$ be a  $U(\C d_0 \op \C s)$ free module of rank one for $\ov{H_4}$. Then they are the classes defined in Theorem 3.2 of \cite{CX}.
\end{theo}
\begin{proof}
   With the help of Theorem \ref{Main theo}, to complete this proof
we need to verify the actions of $d_m$ and $s \ot t^m$ are identical with the actions of Theorem 3.2 of \cite{CX}. \\
{\bf Case I:} Let $p.1 =g(s)\neq 0$ (similar method works if $q.1 \neq 0$).\\  Note that we have $d_n.f(s,d)=f(s,d-n)d_n.1$. Let $d_n.1=f_n(s,d)$. Now consider $[d_n,p].1=0,$ which implies that $g(s)f_n(s,d)=g(s)f_n(s-1,d).$ Since $g(s) \neq 0$, we have $f_n(s,d)=f_n(s-1,d)$, this forces that $deg_s(f_n(s,d))=0.$ Let $d_n.1=f_n(d).$ For all $n \in \Z \setminus\{0\}$ consider the relation,
     \begin{align*}
         \al_1^{n+1}g(s)=p\ot t^{n+1}.1& =[d_n,p\ot t].1 \cr
         &=d_n(\al_1g(s))-p \ot t.(f_n(d)) \cr
         &=\al_1g(s)f_n(d)-\al_1f_n(d-1)g(s) \cr
         &=\al_1g(s)(f_n(d)-f_n(d-1)).
         \end{align*}
The above relation implies that $f_n(d)-f_n(d-1)=\al_1^n$ and hence $f_n(d)$ is a one degree polynomial. Therefore we have  $f_n(d)=\al_1^nd+\mu_n$ for some constant $\mu_n,$ for all $n \in \Z.$ Note that $\mu_0=0$. We know that $s\ot t^k.1=\al_1^ks+\be_k$ for all $k \in \Z$ and $\be_0=0.$\\
{\bf Claim :} $\be_k=0$ for all $k \in \Z.$ Consider the relation 
\begin{align*}
    ks\ot t^{k+n}.1 &=[d_n,s \ot t^k].1 \cr
    &=k\al_1^n(\al_1^ks+\be_k).
\end{align*}
Now comparing both side we have, $\be_{k+n}=\al_1^n\be_k$ for all $k \neq 0$, for all $n \in \Z$. Now putting $n=-k$ and using the fact $\be_0=0$ we have $\be_k=0$ for all $k \in \Z$.

Now we consider the relation $[(n-m)d_{m+n}.1=[d_m,d_n].1$, this implies that 
\begin{align}\label{vireq4.9}
    (n-m)d_{m+n}.1=(\al_1^n(d-m)+\mu_n)(\al_1^{m}d+\mu_{m})-(\al_1^{m}(d-n)+\mu_{m})(\al_1^nd+\mu_n)
\end{align}
Putting $m=-n \neq 0,$ in equation (\ref{vireq4.9}) and equating constant term we have, $\mu_n=-\al_1^{2n}\mu_{-n}$, for all $n \neq 0.$\\
{\bf Claim: } $\mu_n=(n-1)\al_1^{n-2}\mu_2-(n-2)\al_1^{n-1}\mu_1,$ for all $n > 2.$ It is easy to verify from the equation (\ref{vireq4.9}) with $m=1,n=2$ that $\mu_3$ satisfy the property of the claim. Assume the claim is true for $n =k.$ 
Now comparing the constant terms in equation (\ref{vireq4.9}) with $m=1,n=k$, we have, 

\begin{align*}
         (k-1)\mu_{k+1}& =-\al_1^k\mu_1+ k\al_1\mu_k\cr
         &=-\al_1^k\mu_1+ k\al_1[(k-1)\al_1^{k-2}\mu_2-(k-2)\al_1^{k-1}\mu_1] \cr
         &=k(k-1)\al_1^{k-1}\mu_2-(k-1)^2\al_1^k\mu_1 .   \cr
         \end{align*}
Therefore the claim is proved by induction principal. Now putting $m=-1, n=3$ in equation (\ref{vireq4.9}) and comparing the constant terms we have, 
\begin{align*}
         4\mu_{2}& =\mu_{-1}(\mu_3+\al_1^3)-\mu_3(\mu_{-1}-3\al_1^{-1})\cr
         &=\mu_{-1}\al_1^3 +3\mu_3 \al_1^{-1}\cr
         &= -\al_1\mu_1+3\al_1^{-1}(2\al_1\mu_2-\al_1^{2}\mu_1).\cr
         \end{align*}
From the above equation implies that $\mu_2=2\al_1\mu_1.$ Hence we have $\mu_n=n \al_1^{n-1}\mu_1=n\al_1^n\la$, where $\mu_1=\al_1\la,$ for some $\la \in \C$. Then we get $\mu_{-n}=-n\al_1^{-n}\la$ for all $n <0$. Hence for all $n \in \Z$ we have $\mu_n=n\al_1^n\la,$ for some $\la \in \C.$ This proves that actions of $d_m$ and $s \ot t^m$ are identical with the actions of \cite{CX} Theorem 3.2. This completes the proof for case I.\\
{\bf Case II:} Let $p.1=q.1=0.$ In this case $M$ will be Cartan free module for the Lie algebra $span\{ s\ot t^k,d_k, K: k \in \Z\} \cong Vir(0,0).$ Now the result follows from Theorem  \ref{th Vir}. Hence completes the proof.

\end{proof}

{\bf Acknowledgments:} 
For this project the First author is supported by funds of the National Natural Science Foundation of China (Grant No. 12071136) and Science and Technology Commission of Shanghai Municipality (Grant No. 22DZ2229014).

\end{document}